\documentclass[11pt]{article}
\usepackage[margin=1.0in]{geometry}
\usepackage[utf8]{inputenc}
\usepackage{mathrsfs,amssymb,amsfonts,amsthm,amsmath,amsbsy}
\usepackage{stmaryrd}
\usepackage{mathtools}
\usepackage{mathabx}
\usepackage{graphicx}
\usepackage{wrapfig}
\usepackage{caption}
\usepackage{subfig}
\usepackage{dsfont}
\usepackage{float}
\usepackage{upgreek}
\usepackage{listings}
\usepackage{lipsum}
\usepackage[export]{adjustbox}
\usepackage[
backend=bibtex,
style=alphabetic,
citestyle=ieee-alphabetic,
maxnames=10
]{biblatex}
\usepackage{titlesec}
\usepackage{hyperref}
\hypersetup{
	colorlinks,
	citecolor=blue,
	linkcolor=blue,
	urlcolor=green}

\newtheorem{theorem}{Theorem}[section]
\newtheorem{proposition}[theorem]{Proposition}

\newtheorem{lemma}[theorem]{Lemma}

\newtheorem{definition1}[theorem]{Definition}

\newtheorem{remark1}[theorem]{Remark}

\DeclarePairedDelimiter\floor{\lfloor}{\rfloor}

\usepackage[ruled,vlined]{algorithm2e}
\usepackage[font=small]{caption}

\usepackage[ruled,vlined]{algorithm2e}
\usepackage[font=small]{caption}
\usepackage[usenames,dvipsnames]{xcolor}
\usepackage{tikz}
\usepackage{subfig}
\usepackage{dsfont,fancyhdr}
\usetikzlibrary{matrix}

\usepackage{color}

\title{Quenched limit for diffusive biased random walks in random environment}
\author{C.\ Scali\footnote{School of Computation, Information and Technology, Technische Universität München, carlo.scali@tum.de}}

\begin{document}

\maketitle

\begin{abstract}
	We prove that every directionally transient random walk in random i.i.d.\ environment, under condition $(T)_{\gamma}$, which admits an annealed functional limit towards Brownian motion also admits the corresponding quenched limit in $d \ge 2$. We exploit a classical strategy that was introduced by Bolthausen and Sznitman but, with respect to the existing literature, we get almost-optimal bounds on the variance of the quenched expectation of certain functionals of the random walk.
\noindent

\end{abstract}

\section{Introduction}

Random walks in random i.i.d.\ environment (RWRE) are nowadays considered a classical model in probability theory. Their study started more than 50 years ago and has been extensive since. We mention the groundbreaking work of \cite{Solomon} that highlighted that the behaviour of RWRE can be strongly atypical with respect to the classical, symmetric or asymmetric, random walks. Indeed, lots of phenomena with mathematical and physical interest as diffusivity, trapping and aging have been explored. In the present work, we analyse the quenched diffusive behaviour of directionally transient walks in $d \ge 2$.

\subsection{Model and main result}

Let $\{\omega_x\}_{x \in \mathbb{Z}^d}$ be a family of i.i.d.\ random variables supported on $\Omega\coloneqq \mathcal{P}^{\mathbb{Z}^d}$, where $\mathcal{P} \coloneqq \{(p(x))_{x \in \mathbb{Z}^d} \colon p(x) \ge 0, \sum_{x} p(x) = 1\}$. Let $\mathbf{P}$ be the probability measure that generates $\omega \in \Omega$, let $\mathbf{E}$ be the corresponding expectation and $\textbf{Var}$ the variance. We define the quenched law of the random walk $\{X_n\}_{n \in \mathbb{N}_0}$, with $X_0 = x$ almost surely, by setting the one-step jump probabilities to be
\begin{equation*}
	P_{x}^{\omega}\left(X_{n+1} = z \mid X_n = y\right) = \omega_{y}(z).
\end{equation*}
We define the annealed law of the walk $\mathbb{P}_x(\cdot) \coloneqq \mathbf{E}[P_x^{\omega}(\cdot)]$. We will denote with $E^\omega_x$ (resp.\ $\mathbb{E}_x$) to be the expectations w.r.t.\ the quenched (resp.\ annealed) law.

Let us fix a unit vector $v^* \in \mathbb{S}^{d-1}$, we say that $X_n$ is transient in the direction $v^*$ if 
\begin{equation*}
	\mathbb{P}_0\left(\lim_{n \to \infty} X_n \cdot v^* = \infty\right) = 1.
\end{equation*}
It is a well-known fact that if the random walk is directionally transient one can define a regeneration structure. We will denote by $\{\tau_i\}_{i \ge 0}$ the regeneration times that we will define formally in Section~\ref{sec:RegTimes} (with the convention $\tau_0=0$). We also write $\tilde{\tau}_i\coloneqq \tau_{i} - \tau_{i-1}$ for all $i \in \mathbb{N}$.

We retain the assumptions of \cite{BouchetSabotdosSantos}, let us state them:
\begin{description}
	\item[Condition (S)] The walk has bounded steps, i.e.\ $\mathbb{P}(\omega_0(z) = 0) = 1$ for all $z \in \mathbb{Z}^d \colon \|x\| > r_0$ for some universal $r_0 \ge 1$.
	\item[Condition (R)] Non collinearity, i.e.\ the set $\mathcal{J}\coloneqq \{z \colon \mathbf{E}[\omega_{0}(z)]>0\}$ satisfies $\mathcal{J}\not \subseteq \mathbb{R}u$ for any $u \in \mathbb{Z}^d$. Furthermore, $\mathbf{P}(\exists z \colon \omega_{0, 0} + \omega_{0, z} = 1) < 1$. Note that this assumption is very mild as it only requires the RWRE not to be effectively one dimensional.
	\item[Condition $(T)_\gamma$] The walk is transient in the direction $v^* \neq 0$. For some $\gamma \in (0, 1]$ and $c>0$ we have that
	\begin{equation}\label{eqn:ConditionTGamma}
		\mathbb{E}_0 \left[\exp\left(c \sup_{0 \le k \le \tau_1} \|X_k\|^\gamma\right)\right] < \infty.
	\end{equation}
\end{description}
If one further assumes $\mathbb{E}_{0}[\tau_2-\tau_1] < \infty$, by \cite{Sznitman_Zerner, SznitmanEffective} we get that $\mathbb{P}_0$-a.s.\
\begin{equation}\label{eqn:DefDirection}
	\lim_{n \to \infty} \frac{X_n}{n} = v_0 \neq 0,
\end{equation}
for a deterministic $v_0$. We will call $v \coloneqq v_0/\|v_0\|$ the asymptotic direction. Note that the asymptotic direction could be well defined (with a different formula) also beyond the ballistic regime. Let us define the sequence of processes $(B^{(n)}_{t})_{t \ge 0}$ as
\begin{equation}\label{eqn:ScaledProcess}
	B^{(n)}_{t} \coloneqq \frac{X_{\floor{nt}} - v_0 \floor{nt}}{\sqrt{n}}, \quad t \ge 0.
\end{equation}
We often restrict the time horizon of the process to a compact interval $[0, T]$ for an arbitrary $T \in \mathbb{R}_+$. We are now in position to state the main theorem.

\begin{theorem}\label{theo:Main}
	Let us consider a random walk in random environment in dimension $d\ge2$ under conditions (S), (R) and $(T)_\gamma$. Suppose that $\mathbb{E}_0[(\tau_2-\tau_1)^2]<\infty$, then for almost every $\omega \in \Omega$, the sequence of processes $B^{(n)}_{t}$ converges in distribution under $P_0^{\omega}$ towards a centred Brownian motion with a deterministic non-degenerate covariance matrix. The convergence happens in Skorohod's $J_1$-topology on c\`{a}dl\`{a}g functions.
\end{theorem}

\begin{remark1}
	In \cite{BouchetSabotdosSantos} the same result is stated with the stronger requirement on the moments of the regeneration times $\mathbb{E}_0[\tilde{\tau}_2^2 \log(\tilde{\tau}_2)^{p}]<\infty$ for $p = p(\gamma)>0$ large enough (explicit). In the next section, we will comment on how we improve this result and why the one stated above is optimal.
\end{remark1}

\subsection{Motivation and main contribution}
RWRE with a directional drift have been a focus of attention in the last 25 years. We refer to a fundamental series of papers by A.-S.\ Sznitman and co-authors on the topic \cite{Sznitman_Zerner, Zerner, SznitmanSlowdown, SznitmanClass, SznitmanEffective, Szn_serie, Sz1, Sz2}. We would also like to mention that a recurring question in this field is to study Condition $(T)_\gamma$ and its variants (in a broader attempt to characterise ballistic behaviour). Some fundamental papers in this direction are \cite{BergerDrewitzRamirez,FriberghKiousRWRE, Guerra, RodriguezRibeiro}. Many other influential works could be listed here.

The contribution of the present work is twofold. Firstly, our main result Theorem~\ref{theo:Main} is optimal, indeed if $\mathbb{E}_0[\tilde{\tau}_2^2]=\infty$ one cannot expect a (classical) central limit theorem to hold for the process $B^{(n)}_{t}$ under the annealed law. Indeed, a quick inspection of the proof in \cite[(4.6)]{SznitmanSlowdown}, yields that it is necessary to have a CLT for $\{\tilde{\tau}_k^2\}_{k \ge 2}$ and this cannot hold if their variance is infinite. 

Our proof follows the lines of a classical strategy introduced by Bolthausen and Sznitman in \cite{Szn_serie}. This strategy has also been used in several works on biased random walks such as \cite{Berger_Zeitouni, RassoulAghaSeppalainen, Bowditch, QuenchedBiasedRWRC}, as well as in the closest work to this one \cite{BouchetSabotdosSantos}. Some of these works used a martingale difference argument to reduce the problem to controlling the number of intersections of two walks evolving in the same environment, we will follow the same strategy.

However, with respect to previous works, the present paper improves the key variance decay estimate (see Proposition~\ref{prop:VarDecay}) from $\log(n)^{-p}$ (achieved in \cite{BouchetSabotdosSantos}, $p>0$ explicit) to $n^{-1/2+ \varepsilon}$, this despite weakening our hypothesis on the tail of regeneration times. This improvement is our main contribution and is achieved by modifying the argument in two ways:
\begin{enumerate}
	\item  In the first step, as already mentioned, we use a martingale difference argument to re-write the variance of the quenched expectation of certain functionals of the random walk. The goal is to obtain Proposition~\ref{prop:VarDecay}. This is nowadays a classical argument, and typically its end-product is an upper bound for the variance in terms of the number of crossings of two random walks in the same environment. We modify this argument so that the final bound for the variance is \eqref{eqn:BoundWithCrossings}. In the expression \eqref{eqn:BoundWithCrossings} we have, in place of the number of crossings, the product of the increment of the regeneration times associated to a crossing. This somewhat weaker statement is crucial to get a better variance bound as highlighted in the next point.
	\item The event that the two walks cross each other and the time spent during these crossings are almost decorrelated. This is possible because of the markovian structure induced by joint regeneration levels (see Section~\ref{sec:RegLevels} for the definition). Note that, thanks to this decorellation step, we recover the fact that regeneration times associated with crossings have a second moment as the original ones. This step is performed in Section~\ref{section:Decoupling}.
\end{enumerate}
These two steps allow us to significantly reduce the contribution of each intersection to the variance. Informally, one can consider crossings of the two walks to be defects, then what we avoid with our argument is to consider the worst possible defect as typical. We believe this strategy to be crucial to the optimality of Theorem~\ref{theo:Main}. 

Finally, if one aims to prove a similar theorem for processes where the $\mathbb{E}_0[\tilde{\tau}_2^2] = \infty$ (see e.g.\ \cite{Poudevigne, Perrel} where non-Gaussian limits arise), then our strategy would be very helpful. Indeed, if the $\{\tilde{\tau}_i\}_{i \ge 0}$ are in the domain of a stable distribution then the worst possible defect (among the first $n$) would be of the order of the sum $\sum_{i = 1}^n\tilde{\tau}_i$. This is due to the well-known fact that the maximum of $n$ random variables in the domain of a stable distribution (with paramenter $\gamma \in (0, 1)$) is of the order of the sum. We aim to pursue this line of research in a forthcoming work that extends the result of \cite{Kious_Frib, QuenchedBiasedRWRC} to lower dimensions.

\subsection{Proof of the main theorem and outline}\label{proof:Main}

\begin{proof}[Proof of Theorem~\ref{theo:Main}]
	We recall that $\tau_1$ is an almost surely finite random variable. We define
	\begin{equation*}
		\bar{B}^{(n)}_{t} \coloneqq \frac{X_{\floor{nt} + \tau_1} -X_{\tau_1} - v_0 \floor{nt}}{\sqrt{n}}.
	\end{equation*}
    By the finiteness of $\tau_1$ we get that 
    \begin{equation*}
    	\frac{\tau_1}{\sqrt{n}} \to 0,\quad \quad\sup_{t \in [0,1]}\frac{X_{\floor{nt} + \tau_1} -X_{\floor{nt}}}{\sqrt{n}} \le r_0 \frac{\tau_1}{\sqrt{n}} \to 0,
    \end{equation*}
    almost surely (by the bounded step assumption). Hence, by an application of Slutsky's theorem it is sufficient to prove the convergence in distribution of $\bar{B}^{(n)}_{t}$. Let $(W^{(n)}_{t})_{t \in [0, T]}$ be the polygonal interpolation of $\bar{B}^{(n)}_{t}$ so that $W^{(n)} \colon [0, T] \to \mathbb{R}^d$ is a continuous function. We may drop the $t$ in the notation for brevity's sake. Let $F$ be a bounded, Lipschitz function $F \colon C([0, T], \mathbb{R}^d) \to \mathbb{R}$, where $C([0, T], \mathbb{R}^d)$ denotes the set of continuous functions $f \colon [0, T] \to \mathbb{R}^d$ that we endow with the uniform topology denoting the norm with $\|\cdot\|_\infty$. The annealed theorem corresponding to Theorem~\ref{theo:Main} was proved in \cite{SznitmanSlowdown, SznitmanClass}. By \cite[Lemma~4.1]{Szn_serie} the problem of proving Theorem~\ref{theo:Main} reduces to proving that
	\begin{equation*}
		\mathbf{Var}\left( E^\omega_0\left[F(W^{(n)})\right] \right) \le Cn^{-c}.
	\end{equation*}
    The result follows from Proposition~\ref{prop:VarDecay}.
\end{proof}

In Section~\ref{sec:Regeneration} we will recall the regeneration structures built in \cite{BouchetSabotdosSantos}. In Section~\ref{sec:Martingale} we use the martingale difference argument to bound the variance. This will lead to a proof of the fundamental inequality \eqref{eqn:BoundWithCrossings}. Finally in Section~\ref{section:Decoupling} we prove that the r.h.s.\ of \eqref{eqn:BoundWithCrossings} is small enough by using the markovian property of the joint regeneration levels to perform the decoupling step.

\section{Regeneration structure(s)}\label{sec:Regeneration}

\subsection{Regeneration times} \label{sec:RegTimes}

In this section we briefly recall how the regeneration structures are defined and their main properties. We will keep the notation of \cite{BouchetSabotdosSantos} as we will transfer some results from there. Let $\theta_{n}$ be the canonical time shift by $n$ units of time. Let $T_r^{\vec{u}} \coloneqq \inf\{n \in \mathbb{N} \colon X_n \cdot \vec{u} \ge r\}$ (for any $\vec{u} \in \mathbb{S}^{d-1}$). Let also $D^{\vec{u}} \coloneqq \inf\{n \in \mathbb{N} \colon X_n \cdot \vec{u} \le  X_0 \cdot \vec{u}\}$. We fix $v^*$ as in Condition~\ref{eqn:ConditionTGamma}. We set $S_0 \coloneqq 0$, $M_0 = X_0 \cdot v^*$ and $R_0=0$, then we define iteratively the sequence of times
\begin{equation}\label{eqn:DefRegTime}
	\begin{split}
		S_{k+1} &\coloneqq T_{M_{k+1}}^{v^*}, \quad R_{k+1} \coloneqq D^{v^*} \circ \theta_{S_{k+1}}, \\
		 M_{k+1} &\coloneqq \sup\left\{X_n \cdot v^* \colon 0 \le n \le R_{k+1} \right\}. 
	\end{split}
\end{equation}
We set $N \coloneqq \inf\{k \ge 1 \colon S_k < \infty, R_{k} = \infty\}$. Finally, we can define the regeneration times as $\tau_1 \coloneqq S_N$, and for all $n \ge 1 $ set $\tau_{n+1} \coloneqq \tau_1(X_{\cdot}) + \tau_n(X_{\tau_1+\cdot} - X_{\tau_1})$. We will call $\{X_{\tau_n}\}_{n \ge 1}$ regeneration points. 

The fundamental property of regeneration times is that $\{(X_i - X_{\tau_k})_{\tau_k \le i \le \tau_{k+1}}, \tau_{k+1} - \tau_{k}\}_{k \ge 1}$ is a sequence of finite random variables i.i.d.\ under $\mathbb{P}_0$ distributed as $\{(X_i)_{0 \le i \le \tau_{1}}, \tau_{1}\}$ under the law $\mathbb{P}_0(\cdot \mid D^{v^*} = \infty)$. We also note that, by construction, the random walk never backtracks below its starting level in the direction $v^*$ after regenerating.

\subsection{Joint regeneration levels} \label{sec:RegLevels}

From here on $X^{(1)}$ and $X^{(2)}$ will indicate two independent random walks evolving in the same environment, this notation extends to random walk-related quantities (e.g.\ $\{\tau_k^{(1)}\}_{k \ge 0}$ denotes the sequence of regeneration times of the first walk). To lighten the notation, we might omit repeating the superscript ${(1)}$, for example we could write $X^{(1)}_{\tau_1}$ in place of $X^{(1)}_{\tau_1^{(1)}}$. Furthermore, for $x, y \in \mathbb{Z}^d$, we define the quenched law of the two walks as 
\begin{equation*}
	P^\omega_{x, y}\left( X^{(1)} \in \cdot, X^{(2)} \in \cdot \right) \coloneqq P_x^{\omega}\left( X^{(1)} \in \cdot\right) P_y^{\omega}\left( X^{(2)} \in \cdot \right).
\end{equation*}
From this we can define the annealed law as $\mathbb{P}_{x, y}(\cdot) = \mathbf{E}[P^\omega_{x, y}(\cdot)]$. The notations for expectation and variance are extended in the obvious way. In the context of two walks $\theta_{n, k}$ will denote the time shift of the first by $n$ units of time and the second by $k$.

The construction of the joint regeneration structure for the model we consider was performed in \cite{RassoulAghaSeppalainen} and \cite{BouchetSabotdosSantos}. A similar structure is also defined in \cite{QuenchedBiasedRWRC}. We choose a minimalistic presentation to avoid repeating the other works too much.

Firstly, here we reduce to $v^* = e_i$ for some $i = 1, \dots, d$. Indeed, as it is stated in \cite[Proposition 5.1]{BouchetSabotdosSantos}, this is possible as Condition~\eqref{eqn:ConditionTGamma} holds for all $\vec{u} \colon \vec{u} \cdot v > 0$, where $v$, defined below \eqref{eqn:DefDirection}, is the asymptotic direction.

Using this observation we define the space $\mathbb{V}_d \coloneqq \{x \in \mathbb{Z}^d \colon x \cdot v^* = 0\}$. A joint regeneration level is, informally, the first level $\mathcal{L}>0$ such that both random walks $X^{(2)}, X^{(2)}$ regenerate when they hit it.

\begin{definition1}
	We define iteratively the joint regeneration times $(\mu_{k}^{(1)}, \mu_{k}^{(2)})$ as $(\mu_{0}^{(1)}, \mu_{0}^{(2)}) = (0, 0)$ and
	\begin{equation*}
		(\mu_{k+1}^{(1)}, \mu_{k+2}^{(2)}) \coloneqq \inf_{n, m}\left\{(\tau_n^{(1)}, \tau_m^{(2)}) \colon \tau_n^{(1)} \ge \mu_{k}^{(1)}, \tau_m^{(2)} \ge \mu_{k}^{(2)}, X^{(1)}_{\tau_n} \cdot v^* = X^{(2)}_{\tau_m} \cdot v^* \right\}.
	\end{equation*}
    Furthermore, we let $\mathcal{L}_k = X^{(1)}_{\mu_{k}} \cdot v^* = X^{(2)}_{\mu_k} \cdot v^*$ denote the $k$-th joint regeneration level.
\end{definition1}
\noindent We note that, when at a joint regeneration level, the difference $X^{(1)}-X^{(2)}$ lies in the space $\mathbb{V}_d$. Let us define the event $\{\mathcal{D}^\bullet = \infty\} \coloneqq \{D^{v^*, (1)} = D^{v^*, (2)} = +\infty\}$, i.e.\ the event of both walks regenerating. Furthermore, we define the notation $\tilde{\mu}^{(i)}_{k} \coloneqq \mu^{(i)}_{k+1} - \mu^{(i)}_{k}$ for $i = 1, 2$.

We summarise two important results \cite[Lemma~5.2, Lemma~5.5]{BouchetSabotdosSantos}:
\begin{proposition}\label{prop:ShortLevels}
	Under our assumptions, there exists a universal constant $\eta>0$ such that we have that
	\begin{equation*}
		\inf_{x, y \in \mathbb{V}_d} \mathbb{P}_{x, y} \left( \{\mathcal{D}^\bullet = \infty\} \right) \ge \eta.
	\end{equation*}
    Furthermore, we have that for all $p>0$ there exists a constant $C = C_p$ such that
    \begin{equation*}
    	\sup_{x, y \in \mathbb{V}_d} \mathbb{E}_{x, y} \Bigg[ \Big(\sup_{0 \le t \le \mu^{(1)}_1} \|X^{(1)}_t\|\Big)^p \Bigg] < \infty,
    \end{equation*}
    the same result holds symmetrically for $X^{(2)}$.
\end{proposition} 
Finally, we present the fundamental property of joint regeneration levels, that is, the markovian structure that they induce. Let us introduce the sigma-algebras
\begin{equation}\label{eqn:SigmaFields}
	\begin{split}
		\mathcal{A} &\coloneqq \sigma\left(\{X_n^{(1)}\}_{n \ge 0}, \{X_n^{(2)}\}_{n \ge 0}\right),\\
		\mathcal{B} &\coloneqq \sigma\left(\omega_x \colon x \in \{X_n^{(1)}\}_{n \ge 0} \cap \{X_n^{(2)}\}_{n \ge 0} \right), \\
		\Sigma_k &\coloneqq \sigma\left(\mathcal{L}_0, \dots, \mathcal{L}_k,\{X_{n \wedge \mu^{(1)}_k}^{(1)}\}_{n \ge 0}, \{X_{n \wedge \mu^{(2)}_k}^{(2)}\}_{n \ge 0}, \omega_x \colon x \cdot v^* < \mathcal{L}_k \right).
	\end{split}
\end{equation}
\begin{proposition}\label{prop:Markov}
	Let $f, g, h$ be bounded function measurable, respectively, with respect to $\mathcal{A}, \mathcal{B}, \Sigma_k$, then, for all $x, y \in \mathbb{Z}^d$
	\begin{equation*}
		\mathbb{E}_{x, y} \left[ f\left(X_{\mu^{(1)}_k + \cdot}^{(1)}, X_{\mu^{(2)}_k + \cdot}^{(2)}\right) g \circ \theta_{\mu^{(1)}_k, \mu^{(2)}_k} h \right] = \mathbb{E}_{x, y} \left[ h \mathbb{E}_{X^{(1)}_{\mu_k}, X^{(2)}_{\mu_k}} \left[ f\left(X_{\cdot}^{(1)}, X_{\cdot}^{(2)} \right) g \mid \mathcal{D}^\bullet = \infty\right] \right].
	\end{equation*}
\end{proposition}

\begin{proof}
	The proof follows the same steps used in the proof of \cite[Proposition 7.7]{RassoulAghaSeppalainen}, one only needs to take expectations of the functions in place of the probabilities considered there. We observe that the result implies \cite[Proposition 7.7]{RassoulAghaSeppalainen}.
\end{proof}

\section{Martingale difference} \label{sec:Martingale}

We consider the notation set in the proof of Theorem~\ref{theo:Main} in Section~\ref{proof:Main}. Without loss of generality, we will assume from here on that $|F(f)| \le 1$ for all $f \in C([0, T], \mathbb{R}^d)$ and that $|F(f) - F(g)| \le \|f- g\|_{\infty}$ for all $f, g \in C([0, T], \mathbb{R}^d)$. Most of the time we reduce, without loss of generality, to the case $T=1$. Note also that $\bar{B}^{(n)}_{t}$ has the same law as $B^{(n)}_{t}$ under $\mathbb{P}_{0}(\cdot \mid D^{v^*} = \infty)$.

\begin{proposition}\label{prop:VarDecay}
	There exists two absolute constants $C, c>0$ such that for all $n \in \mathbb{N}$ we have
	\begin{equation*}
		\mathbf{Var}\left( E^\omega_0\left[F(W^{(n)})\right] \right) \le Cn^{-c}.
	\end{equation*}
More precisely, one can choose $c<1/2$ as close to $1/2$ as desired.
\end{proposition}

\begin{proof}
	\noindent \textbf{Finite box.} Let us define the event
	\begin{equation*}
		O_n \coloneqq \{\sup_{0 \le k \le \tau_1} \|X_k\| \le n^{\varepsilon}\}.
	\end{equation*}
    Markov inequality and \eqref{eqn:ConditionTGamma} readily give that $\mathbb{P}_{0}(O_n^c) \le Cn^{-M}$ for any $M>0$. Hence, exploiting the boundedness of $F$, we get that
    \begin{equation*}
    	\textbf{Var}\left( E^\omega_0\left[F(W^{(n)})\right] \right) \le \textbf{Var}\left( E^\omega_0\left[F(W^{(n)})\mathds{1}_{\{O_n\}}\right] \right) + Cn^{-M}.
    \end{equation*}
	\noindent \textbf{Doob's martingale.}
	Clearly
	\begin{equation*}
		\textbf{Var}\left( E^\omega_0\left[F(W^{(n)})\mathds{1}_{\{O_n\}}\right] \right) = \textbf{E}\left[ \left(E^\omega_0\left[F(W^{(n)})\mathds{1}_{\{O_n\}}\right] -  \mathbb{E}_0\left[F(W^{(n)})\mathds{1}_{\{O_n\}}\right]\right)^2\right].
	\end{equation*}
    Let us fix the notation $d_n \coloneqq n^{1+\varepsilon}$, we remark that $W^{(n)}$ depends only the first $\tau_1 + n$ steps of the random walk. Let $B(0,d_n)$ be the $L^1$ ball centred at the origin and of radius $d_n$, we introduce some fixed ordering of $B(0, d_n)$ such that all $z$ that precede $x$ in the ordering are such that $z \cdot v^* \le x \cdot v^*$ for all $x \in B(0, d_n)$. Let $p(z)$ denote the predecessor of $z$ in such ordering. Finally let $\mathcal{G}_z \coloneqq \sigma( \{\omega_x\}_{x \le z})$, where the inequality symbol should be interpreted in the ordering sense. By applying a standard Doob martingale argument to the filtration $\{\mathcal{G}_z\}_{z \in B(0, d_n)}$ we get
	\begin{equation}\label{eqn:BoundVarIntersectionClock}
		\textbf{Var}\left( E^\omega_0\left[F(W^{(n)})\right] \right) = \sum_{z \in B(0,d_n)} \mathbf{E}\left[ \left(\Delta_{z}^n \right)^2\right],
	\end{equation}
where
\begin{equation*}
	\Delta_{z}^n \coloneqq \mathbf{E}\left[ E_0^\omega\left[ F(W^{(n)})\mathds{1}_{\{O_n\}} \right] | \mathcal{G}_{z} \right] - \mathbf{E}\left[ E_0^\omega\left[ F(W^{(n)})\mathds{1}_{\{O_n\}} \right] | \mathcal{G}_{p(z)} \right].
\end{equation*}
Let us also notice that
\begin{equation*}
	\mathbf{E}\left[ E_0^\omega\left[ F(W^{(n)})\mathds{1}_{\{O_n\}} \mathds{1}_{\{X \text{ visits }z \}^c} \right] | \mathcal{G}_{z} \right]  = \mathbf{E}\left[ E_0^\omega\left[ F(W^{(n)})\mathds{1}_{\{O_n\}}\mathds{1}_{\{X \text{ visits }z \}^c} \right] | \mathcal{G}_{p(z)} \right],
\end{equation*}
hence we can restrict to studying the quantity of interest on the event $\{X \text{ visits }z \}$.

\noindent \textbf{Non uniform bound.} We now aim to construct a quantity that approximates $W^{(n)}$ but whose distribution does not depend on the environment at $z$. For two paths $\pi_1$ and $\pi_2$ of length $m$ and $k$ respectively, let $\pi_1 \circ \pi_2$ denote the concatenation, i.e.\ for $ 0 \le s \le m+k$ 
\begin{equation}\label{eqn:ConcatenationPaths}
	\pi_1 \circ \pi_2(s) = 
	\begin{cases}
		\pi_1(s) \quad &\text{if } s \le m,\\
		\pi_1(m) + \pi_{2}(s - m) \quad &\text{if } s \in \llbracket m, k+m \rrbracket.
	\end{cases}
\end{equation}
Let $\{X\}_{a}^{b}$ denote the trajectory of the random walk between times $a$ and $b$. We recall the shortcut $T_{z\cdot v^*} = \inf\{k \ge 0\colon X_{k} \cdot v^* \ge z \cdot v^*\}$. Let us set the notation $\tau_z^{-} = T_{z\cdot v^*}$, note that this quantity is almost surely well defined by the transience of the random walk. Furthermore, let $\tau^+_z \coloneqq \inf\{ \tau_{k} \colon X_{\tau_{k}} \cdot v^* > z \cdot v^*\}$, i.e.\ the first regeneration time that happens to the ``right'' of $z$. Furthermore, let $\tau_1^\bullet$ be the first candidate regeneration time in $\{X\}_{0}^{\tau_z^-}$. In formulas, $\tau_1^\bullet = \tau_1^\bullet(z) \coloneqq \inf\{ k \le T_{z\cdot v^*} \colon X_{k}\cdot v^* > X_{j}\cdot v^* \text{ for all }j\le k, D^{v^*}\circ \theta_k + k \ge T_{z} \}$.

We define the approximating quantity as $\{\tilde{X}\}_{0}^n = \{X\}_{0}^{\tau_z^-} \circ \{X - X_{\tau_z^+} \}_{\tau_z^+}^{n+ \tau_1^\bullet}$ and as a consequence we define
\begin{equation}\label{eqn:GluedTrajectory}
	\tilde{B}^{(n)}_{t} \coloneqq \frac{\tilde{X}_{\floor{nt} + \tau^\bullet_1} -X_{\tau^\bullet_1} - v_0 \floor{nt}}{\sqrt{n}}, \quad t \ge 0.
\end{equation}
Finally, one can define $(\tilde{W}^{(n)}_t)_{t \in [0, 1]}$ as the polygonal interpolation of \eqref{eqn:GluedTrajectory}. We make the following observations on $\{\tilde{X}\}_{0}^n$:
\begin{enumerate}
	\item The trajectory $\{\tilde{X}\}_{0}^n$ is independent of the environment at $z$. Indeed, the trajectory before $\tau_z^- \le T_{z\cdot v^*} \le T_{z}$ is clearly independent, and the trajectory after $\tau^+_z$ is also independent by the fundamental property of regeneration times (see Section~\ref{sec:RegTimes}). 
	\item We also claim that
	\begin{equation*}
		\sup_{t \in [0, 1]} \left| W^{(n)}_t - \tilde{W}^{(n)}_t \right| \le \frac{r_0}{n^{1/2}} \sum_{j = 2}^{n+1} \mathds{1}_{\{z, j\}} \tilde{\tau}_{j} + \frac{r_0}{n^{1/2}}(\tau^{-}_z - \tau^\bullet_1)\mathds{1}_{\{z, 1\}},
	\end{equation*}
    where $\{z, j\} \coloneqq \{ z \in \{X\}_{\tau_{j-1}}^{\tau_{j}-1}\}$. In words, the difference between the two trajectories is upper bounded by the re-scaled increment of the (unique) regeneration time in which the hitting of site $z$ happens. The first regeneration time is treated separately as its distribution is not the one of an increment between successive regeneration times.
\end{enumerate}
To support the second claim we observe that, after hitting $z$ (let's assume for now that $\{z, j\}$ holds for $j \ge 2$), the walk can erase the ``candidate regeneration points'' that it already met. Hence the difference $\tau_z^+ - \tau_z^- \le \sum_{j = 2}^{n+1} \mathds{1}_{\{z, j\}} \tilde{\tau}_{j}$, the claim follows. Note as well that, by construction, we have that deterministically $\tau_{n+1} > n + \tau_1$. On the other hand, if $\mathds{1}_{\{z, 1\}}$ happens, then the quantity $(\tau^{-}_z - \tau^\bullet_1)$ does appear in $\tilde{W}^{(n)}_t$ (by its construction), but not in $W^{(n)}_t$ as $\tau_1 \ge \tau^{-}_z$.

We set the notation
\begin{equation*}
	\mathbf{E}_z\left[ E_0^\omega\left[ \cdot \right] \right] \coloneqq \mathbf{E}\left[ E_0^\omega\left[ \cdot \right] | \mathcal{G}_{z} \right].
\end{equation*}
We also set in the same way the notation $\mathbf{E}_{p(z)}\left[ E_0^\omega\left[ \cdot \right] \right]$. From the considerations above we observe that
\begin{equation}\label{eqn:EqualityConditional}
	\mathbf{E}_z\left[ E_0^\omega\left[ F(\tilde{W}^{(n)}) \right] \right] = \mathbf{E}_{p(z)}\left[ E_0^\omega\left[ F(\tilde{W}^{(n)}) \right] \right].
\end{equation}

\noindent \textbf{Intermezzo on $\tau_1$.} We now briefly focus on describing the quantity $(\tau^{-}_z - \tau^\bullet_1)$ on the event $\{z, 1\}$. Note that both these quantities were defined as candidate regeneration times. Hence, both $X_{\tau^\bullet_1}$ and $X_{\tau^{-}_z}$ are associated to maxima of the random walk in the direction $v^*$. Then, if taken under the measure $\mathbb{P}_0$, the quantity $(\tau^{-}_z - \tau^\bullet_1)$ can be dominated by $T_{z \cdot v^*}$ taken under the measure $\mathbb{P}_0( \cdot \mid D^{v^*} \ge T_{z \cdot v^*})$. Since we consider all these quantities on the event $O_n$, we are able to dominate 
\begin{equation*}
	\mathbb{E}_{0}[(\tau^{-}_z - \tau^\bullet_1)\mathds{1}_{O_n}] \le \sup_{0 \le z\cdot v^* \le n^{\varepsilon}} \mathbb{E}_0[ T_{z\cdot v^*}| D^{v^*} \ge T_{z\cdot v^*}].
\end{equation*}
We will show how the measure $\mathbb{P}_0$ appears and how to apply this domination in the following steps.

\noindent \textbf{Jensen's inequality.} We note that, by the Lipschitz and boundedness of $F$, we can write
\begin{equation*}
	\left|F(W^{(n)}) - F(\tilde{W}^{(n)}) \right| \le  r_0\sum_{j = 2}^{n+1} \mathds{1}_{\{z, j\}} \left(\frac{\tilde{\tau}_{j}}{n^{1/2}} \wedge 1\right)+ (\frac{r_0}{n^{1/2}}(\tau^{-}_z - \tau^\bullet_1)\wedge 1)\mathds{1}_{\{z, 1\}}.
\end{equation*}
Putting together the estimates above and the definition of $\Delta_{z}^n$, we get terms of the form
\begin{equation*}
	\sum_{z \in B(0, d_n)}\mathbf{E}\left[\mathbf{E}_z\left[ E_0^\omega\left[ r_0\sum_{j = 2}^{n+1} \mathds{1}_{\{z, j\}} \left(\frac{\tilde{\tau}_{j}}{n^{1/2}} \wedge 1\right)+ \frac{r_0}{n^{1/2}}((\tau^{-}_z - \tau^\bullet_1)\wedge 1)\mathds{1}_{\{z, 1\}}  \right] \right]^2 \right].
\end{equation*} 
Note that there is also a corresponding term with $\mathbf{E}_{p(z)}$ (we can get rid of the cross terms in the square by using the inequality $(a+b)^2 \le 2a^2 + 2b^2$). We use Jensen's inequality, the tower property  to get the upper bound
\begin{equation}\label{eqn:DeltaModified}
	4\sum_{z \in B(0, d_n)}\mathbf{E}\left[ E_0^\omega\left[ \sum_{j = 2}^{n+1} \mathds{1}_{\{z, j\}} \left(\frac{\tilde{\tau}_{j}}{n^{1/2}} \wedge 1\right) \right]^2\right] + 4\sum_{z \in B(0, d_n)} \mathbb{E}_0\left[(\frac{r_0}{\sqrt{n}}(\tau^{-}_z - \tau^\bullet_1)\wedge 1)^2\mathds{1}_{\{z, 1\}}\right].
\end{equation} 
The second term is handled easily as we notice that if $\|z\|\ge n^\varepsilon$ we have $\mathbb{E}_0[\mathds{1}_{\{z, 1\}}]\le n^{-M}$ for $M>0$ arbitrarily large, hence
\begin{equation*}
	4 \sum_{z \in B(0, d_n)} \mathbb{E}_0\left[(\frac{r_0}{\sqrt{n}}(\tau^{-}_z - \tau^\bullet_1)\wedge 1)^2\mathds{1}_{\{z, 1\}}\right]\le Cn^{d\varepsilon - 1}\sup_{0 \le z\cdot v^* \le n^{\varepsilon}} \mathbb{E}_0[ T^2_{z\cdot v^*}| D^{v^*} \ge T_{z\cdot v^*}] + n^{-M+2d}.
\end{equation*}
We apply Lemma~\ref{lemma:Extra2} to conclude that the second term in \eqref{eqn:DeltaModified} is smaller than $Cn^{-1+2d\varepsilon}$ and therefore negligible as $\varepsilon>0$ can be chosen suitably small.

\noindent \textbf{Intersections of two walks.} We now use the promised strategy of analysing two walks in the same environment. Indeed, with hopefully obvious notation, we get that the first summand in \eqref{eqn:DeltaModified} is equal (up to a universal multiplicative constant) to
\begin{equation*}
	\sum_{z \in B(0, d_n)}\mathbf{E}\left[ E_0^\omega\left[ \sum_{j = 2}^{n+1} \mathds{1}^{(1)}_{\{z, j\}} \left(\frac{\tilde{\tau}^{(1)}_{j}}{n^{1/2}} \wedge 1\right) \right] E_{0}^\omega\left[ \sum_{j = 2}^{n+1} \mathds{1}^{(2)}_{\{z, j\}} \left(\frac{\tilde{\tau}^{(2)}_{j}}{n^{1/2}} \wedge 1\right) \right]\right].
\end{equation*} 
We note that, by re-writing the quantity in the last display using the notation set in Section~\ref{sec:RegLevels}, we obtain
\begin{equation}\label{eqn:PenultimateStepSum}
	\sum_{z \in B(0, d_n)}\mathbf{E}\left[ E_{0, 0}^\omega\left[ \sum_{j, k = 2}^{n+1} \mathds{1}^{(1)}_{\{z, j\}} \left(\frac{\tilde{\tau}^{(1)}_{j}}{n^{1/2}} \wedge 1\right) \mathds{1}^{(2)}_{\{z, k\}} \left(\frac{\tilde{\tau}^{(2)}_{k}}{n^{1/2}} \wedge 1\right) \right]\right].
\end{equation} 
Before proceeding, we define the random sets 
\begin{equation}\label{eqn:RandomIndeces}
	\begin{split}
		J_n &\coloneqq \left\{j = 2, \dots ,n +1 \colon \{X^{(1)}\}_{\tau_j}^{\tau_{j+1} - 1} \cap \{X^{(2)}\}_{\tau_1}^{\tau_{n+1}} \neq \emptyset\right\} \\
		K_n(j) &\coloneqq \left\{k = 2, \dots ,n+1 \colon \{X^{(2)}\}_{\tau_k}^{\tau_{k+1} - 1} \cap \{X^{(1)}\}_{\tau_j}^{\tau_{j+1}-1} \neq \emptyset\right\}.
	\end{split}
\end{equation}
We remark that, even if the definition above is not symmetric, the index of the walk does not play any meaningful role and it could be reversed without any issues.

Let us introduce the event 
\begin{equation*}
	G_n \coloneqq \left\{\sup_{i \in\{1, 2\}, j = 2, \dots, n+1} \|\{X^{(i)} - X^{(i)}_{\tilde{\tau}_{j-1}}\}_{\tilde{\tau}_{j-1}}^{\tilde{\tau}_{j}}\|_\infty \le n^{\varepsilon/d} \right\}.
\end{equation*}
It is immediate using \eqref{eqn:ConditionTGamma}, union bound and Markov inequality to get $\mathbb{P}_{0, 0}(G_n^c) \le n^{-M}$ for arbitrary large $M>0$. Hence, since the sum inside \eqref{eqn:PenultimateStepSum} is bounded by $n$ and $|B(0, d_n)|\le C_dn^{d+1}$, we get 
\begin{equation*}
		\sum_{z \in B(0, d_n)}\mathbf{E}\left[ E_{0, 0}^\omega\left[ \sum_{j, k = 2}^{n+1} \mathds{1}^{(1)}_{\{z, j\}} \left(\frac{\tilde{\tau}^{(1)}_{j}}{n^{1/2}} \wedge 1\right) \mathds{1}^{(2)}_{\{z, k\}} \left(\frac{\tilde{\tau}^{(2)}_{k}}{n^{1/2}} \wedge 1\right) \mathds{1}_{\{G_n\}} \right]\right] + Cn^{-M}.
\end{equation*}
We drop the negligible term $Cn^{-M}$ from now on. We make the following observations:
\begin{enumerate}
	\item Two indices $j, k$ in the sum \eqref{eqn:PenultimateStepSum} are present only if there exists $z \in B(0, d_n)$ on which the two walks meet.
	\item A pair of indices $j \in J_n$ and $k \in K_n(j)$ is counted at most $n^{\varepsilon}$ times when summing over $z \in B(0, d_n)$. Indeed, this is guaranteed by the good event $G_n$.
\end{enumerate}
Hence, we get that the quantity in \eqref{eqn:PenultimateStepSum} is bounded from above by 
\begin{equation*}
	n^{\varepsilon}\mathbf{E}\left[ E_{0, 0}^\omega\left[ \sum_{j \in J_n, k = K_n(j)} \left(\frac{\tilde{\tau}^{(1)}_{j}}{n^{1/2}} \wedge 1\right) \left(\frac{\tilde{\tau}^{(2)}_{k}}{n^{1/2}} \wedge 1\right) \right]\right].
\end{equation*} 
Note that here we cannot pull the sum out of the expectation as it is random itself. We further bound the sum, let us define a set of random joint regeneration levels as follows
\begin{equation}\label{eqn:RandomIndecesJRL}
    \mathrm{JRLC}(n) = \left\{ i=1, \dots, n \colon \{X^{(1)}\}_{\mu_i}^{\mu_{i+1}} \cap \{X^{(2)}\}_{\mu_i}^{\mu_{i+1}} \neq \emptyset \right\}.
\end{equation}
In words, these are the joint regeneration slabs in which a meeting occurs. By the construction of joint regeneration levels of Section~\ref{sec:RegLevels}, it follows that we get the upper bound
\begin{align}
	\mathbf{Var}\left( E^\omega_0\left[F(W^{(n)})\right] \right) &\le n^{\varepsilon}\mathbf{E}\left[ E_{0, 0}^\omega\left[ \left(\frac{\mu^{(1)}_{1} - \tau^{(1)}_1}{n^{1/2}} \wedge 1\right) \left(\frac{\mu^{(2)}_{1} - \tau^{(2)}_1}{n^{1/2}} \wedge 1\right) \right]\right]\nonumber\\
	&+n^{\varepsilon}\mathbf{E}\left[ E_{0, 0}^\omega\left[ \sum_{i \in \mathrm{JRLC}(n)} \left(\frac{\tilde{\mu}^{(1)}_{i}}{n^{1/2}} \wedge 1\right) \left(\frac{\tilde{\mu}^{(2)}_{i}}{n^{1/2}} \wedge 1\right) \right]\right] + Cn^{-1/2}.\label{eqn:BoundWithCrossings}
\end{align} 
The result follows by a direct application of Lemma~\ref{lemma:Extra1} and Proposition~\ref{prop:DecorrelationBound}. We note that the $c$ of the statement can be quantified as $1/2 - \delta$ for an arbitrary small $\delta>0$.
\end{proof}

\section{Decoupling through JRL} \label{section:Decoupling}

For $n \in \mathbb{N}$ and $\varepsilon>0$ define the set of indices
\begin{equation*}
	\mathrm{JRL}^{\le}(n, \varepsilon) \coloneqq\left\{k =1, \dots, n \colon \|X^{(1)}_{\mu_k} - X^{(2)}_{\mu_k}\|\le n^{\varepsilon}\right\}.
\end{equation*}
Furthermore, we set $\mathrm{JRL}^{>}(n, \varepsilon) = \{1, \dots, n\}\setminus \mathrm{JRL}^{\le}(n, \varepsilon)$. We may write $\mathrm{JRL}^{\le}(n)$ in place of $\mathrm{JRL}^{\le}(n, \varepsilon)$. The first result of this section shows that the number of intersections for the two walks is upper bounded by $n^{1/2+c\varepsilon}$.

\begin{proposition}\label{prop:ExpectedCrossings}
	For all $\varepsilon>0$ we have that, for all $n$ large enough
	\begin{equation*}
		\mathbb{E}_{0, 0}[|\mathrm{JRL}^{\le}(n, \varepsilon)|] \le n^{1/2 + 13 \varepsilon}.
	\end{equation*}
\end{proposition}
\begin{proof}
	The result is proved in \cite{RassoulAghaSeppalainen} and \cite{BouchetSabotdosSantos}. Indeed, in the latter, it is proved that the RWRE considered here satisfies the assumptions of \cite[Theorem~A.1]{RassoulAghaSeppalainen}. 
	
	An intermediate step of the proof of \cite[Theorem~A.1]{RassoulAghaSeppalainen} is equation \cite[(A.4)]{RassoulAghaSeppalainen}. Note that the $13\varepsilon$ is stated at the end of the proof (end of \cite[Appendix~A]{RassoulAghaSeppalainen}). Observe that $\varepsilon>0$ can be chosen anyway arbitrarily small in our context (thanks to the assumption of Condition $(T)_\gamma$). 
\end{proof}

The next proposition, the main of this section, shows that the increment of the joint regeneration times associated to intersections is negligible in the limit.

\begin{proposition}\label{prop:DecorrelationBound}
	For all $\varepsilon>0$ there exist a positive constant $C>0$, such that, for all $n \in \mathbb{N}$ 
	\begin{equation*}
		\mathbb{E}_{0, 0}\left[ \sum_{i \in \mathrm{JRLC}(n)} \left(\frac{\tilde{\mu}^{(1)}_{i}}{n^{1/2}} \wedge 1\right) \left(\frac{\tilde{\mu}^{(2)}_{i}}{n^{1/2}} \wedge 1\right) \right]\le C n^{-1/2 + \varepsilon}.
	\end{equation*}
\end{proposition}

\begin{proof}
	The idea of the proof is to exploit the markovian structure of joint regeneration levels. We introduce the event
	\begin{equation}\label{eqn:GoodEvent}
		E_n \coloneqq \left\{ \sup_{i = 1, \dots, n} \max\left(\Big\|\{X^{(1)}\}_{\mu_i}^{{\mu_{i+1}}} - X^{(1)}_{\mu_i}\Big\|_\infty, \Big\|\{X^{(2)}\}_{\mu_i}^{{\mu_{i+1}}}- X^{(2)}_{\mu_i}\Big\|_\infty \right) \le \frac{n^{\varepsilon}}{2} \right\}.
	\end{equation}
    Applying a union bound, Proposition~\ref{prop:ShortLevels} and Markov inequality we get that $\mathbb{P}_{0, 0}(E_n) \ge 1 - n^{-M}$ for arbitrary $M>0$. On $E_n$ we have that $\mathrm{JRLC}(n) \subseteq\mathrm{JRL}^{\le}(n)$ since the two walks cannot meet as they start from a distance larger than $n^{\varepsilon}$ and get as far as $n^{\varepsilon}/2$ from their respective starting points. To shorten our notation, let us denote
	\begin{equation*}
		\hat{\mu}^{(1)}_i(n) \coloneqq \left(\frac{\tilde{\mu}^{(1)}_{i}}{n^{1/2}} \wedge 1\right),
	\end{equation*}
    and define similarly $\hat{\mu}_i^{(2)}(n)$. Let us re-write the main quantity
	\begin{align}
			\mathbb{E}_{0, 0}\left[ \sum_{i \in \mathrm{JRLC}(n)} \hat{\mu}^{(1)}_i(n) \hat{\mu}^{(2)}_i(n) \right] 
			&= \mathbb{E}_{0, 0} \left[ \sum_{i = 1}^{n} \mathds{1}_{\{i \in \mathrm{JRLC}(n)\}} \hat{\mu}^{(1)}_i(n) \hat{\mu}^{(2)}_i(n) \right]\nonumber \\
			&= \mathbb{E}_{0, 0}\left[ \sum_{i = 1}^{n} \mathds{1}_{\{i \in \mathrm{JRLC}(n)\}} \mathds{1}_{\{i \in \mathrm{JRL}^{\le}(n)\}} \hat{\mu}^{(1)}_i(n) \hat{\mu}^{(2)}_i(n) \right]\label{eqn:TermClose}\\
			&+\mathbb{E}_{0, 0}\left[ \sum_{i = 1}^{n} \mathds{1}_{\{i \in \mathrm{JRLC}(n)\}} \hat{\mu}_i^{(1)}(n) \hat{\mu}_i^{(2)}(n)  \mathds{1}_{\{\mathrm{JRLC}(n) \not\subseteq\mathrm{JRL}^{\le}(n)\}} \right].\nonumber
	\end{align}
It is immediate to see that $\hat{\mu}_i^{(1)}(n) \hat{\mu}_i^{(2)}(n) \le 1$, hence, the second term is easily bounded
\begin{align*}
	\mathbb{E}_{0, 0}\left[ \sum_{i = 1}^{n} \mathds{1}_{\{i \in \mathrm{JRLC}(n)\}} \hat{\mu}_i^{(1)}(n) \hat{\mu}_i^{(2)}(n)  \mathds{1}_{\{\mathrm{JRLC}(n) \not\subseteq\mathrm{JRL}^{\le}(n)\}} \right]
	&\le n \mathbb{E}_{0, 0}\left[ \sum_{i = 1}^{n} \mathds{1}_{\{i \in \mathrm{JRLC}(n)\}} \mathds{1}_{\{\mathrm{JRLC}(n) \not\subseteq\mathrm{JRL}^{\le}(n)\}} \right]\\
	& \le n^2 \mathbb{P}_{0,0}\left( \mathrm{JRLC}(n) \not\subseteq\mathrm{JRL}^{\le}(n) \right)\\
	&\le n^2\mathbb{P}_{0,0}\left( E_n^c \right) \le n^{-M+2}.
\end{align*}
The term \eqref{eqn:TermClose} is more involved and requires Lemma~\ref{lemma:TailJRL}. The basic idea is to use the tower property ``backwards'' in time and exploit the fact that while the event $\mathds{1}_{\{i \in \mathrm{JRL}^{\le}(n)\}}$ depends on what happens before $\mathcal{L}_i$, the two quantities $\hat{\mu}_i^{(1)}(n) \hat{\mu}_i^{(2)}(n)$ depend on what happens after. 

Recall that $\Sigma_n$ defined in \eqref{eqn:SigmaFields} and that represents the information generated by the walk before $\mathcal{L}_n$. Using the tower property we are allowed to write
\begin{align*}
	\mathbb{E}_{0, 0}\left[ \sum_{i = 1}^{n}\mathds{1}_{\{i \in \mathrm{JRL}^{\le}(n)\}} \mathds{1}_{\{i \in \mathrm{JRLC}(n)\}} \hat{\mu}_i^{(1)}(n) \hat{\mu}_i^{(2)}(n) \right] 
	&= \mathbb{E}_{0, 0}\left[ \mathbb{E}_{0, 0}\left[ \sum_{i = 1}^{n} \mathds{1}_{\{i \in \mathrm{JRLC}(n)\}} \mathds{1}_{\{i \in \mathrm{JRL}^{\le}(n)\}} \hat{\mu}_i^{(1)}(n) \hat{\mu}_i^{(2)}(n) \big| \Sigma_n \right] \right] \\
	& \le \mathbb{E}_{0, 0}\left[ \sum_{i = 1}^{n-1} \mathds{1}_{\{i \in \mathrm{JRLC}(n)\}} \mathds{1}_{\{i \in \mathrm{JRL}^{\le}(n)\}} \hat{\mu}^{(1)}(n) \hat{\mu}^{(2)}(n)  \right]\\ &+  \mathbb{E}_{0, 0}\left[ \mathds{1}_{\{i \in \mathrm{JRL}^{\le}(n)\}} \mathbb{E}_{0, 0}\left[ \hat{\mu}_i^{(1)}(n) \hat{\mu}_i^{(2)}(n) \big| \Sigma_n \right] \right].
\end{align*}
Let us focus on the last term, applying the Markov property of regeneration levels stated in Proposition~\ref{prop:Markov} we can write it as
\begin{align*}
	\mathbb{E}_{0, 0}\left[ \sum_{i = 1}^{n-1} \mathds{1}_{\{i \in \mathrm{JRLC}(n)\}} \mathds{1}_{\{i \in \mathrm{JRL}^{\le}(n)\}} \hat{\mu}_i^{(1)}(n) \hat{\mu}_i^{(2)}(n)  \right]+  \mathbb{E}_{0, 0}\left[ \mathds{1}_{\{n \in \mathrm{JRL}^{\le}(n)\}} \mathbb{E}_{X^{(1)}_{\mu_{n}}, X^{(2)}_{\mu_{n}}}\left[ \hat{\mu}^{(1)}_1(n) \hat{\mu}_1^{(2)}(n)  \mid \mathcal{D}^{\bullet} = \infty \right] \right].
\end{align*}
Finally, applying Lemma~\ref{lemma:TailJRL} (note that the bound there works uniformly over the starting points) we get
\begin{align*}
	\mathbb{E}_{0, 0}&\left[ \sum_{i = 1}^{n}\mathds{1}_{\{i \in \mathrm{JRL}^{\le}(n)\}} \hat{\mu}_i^{(1)}(n) \hat{\mu}_i^{(2)}(n) \right]\\ 
	& \le \mathbb{E}_{0, 0}\left[ \sum_{i = 1}^{n-1} \mathds{1}_{\{i \in \mathrm{JRLC}(n)\}} \mathds{1}_{\{i \in \mathrm{JRL}^{\le}(n)\}} \hat{\mu}_i^{(1)}(n) \hat{\mu}_i^{(2)}(n)  \right]+  \mathbb{P}_{0, 0}\left(n \in \mathrm{JRL}^{\le}(n)\right)n^{-1+2\varepsilon}.
\end{align*}
Iterating this step $n$ times and applying Proposition~\ref{prop:ExpectedCrossings}, we get the upper bound for \eqref{eqn:TermClose}
\begin{equation*}
	n^{-1+2\varepsilon} \mathbb{E}_{0, 0}[|\mathrm{JRL}^{\le}(n)|] \le n^{-1/2 + 15\varepsilon}.
\end{equation*}
This concludes the proof as $\varepsilon$ is arbitrary.

\end{proof}

The next lemma provides a rigorous way to quantify the decorrelation step.

\begin{lemma}\label{lemma:TailJRL}
	There exists a positive constant $C>0$, such that, uniformly over all $U_1, U_2 \in \mathbb{V}_d$ for all $n \in \mathbb{N}$ 
	\begin{equation*}
		\mathbb{E}_{U_1, U_2}\left[ \left(\frac{\tilde{\mu}^{(1)}_{1}}{n^{1/2}} \wedge 1\right) \left(\frac{\tilde{\mu}^{(2)}_{1}}{n^{1/2}} \wedge 1\right) \Big| \mathcal{D}^{\bullet} = \infty\right]\le Cn^{2\varepsilon - 1}.
	\end{equation*}
\end{lemma}
\begin{proof}
	The proof exploits the estimate of the probability of \eqref{eqn:GoodEvent}. Indeed, let $\tilde{E}_n \coloneqq \{\mu_1^{(1)} \le \tau^{(1)}_{n^{\varepsilon}}, \mu_1^{(2)} \le \tau^{(2)}_{n^{\varepsilon}}\}$. It is straightforward to show that $\mathbb{P}_{U_1, U_2}(\tilde{E}_n | \mathcal{D}^{\bullet} = \infty) \ge \mathbb{P}_{U_1, U_2}(E_n) \ge 1 -n^{-p}$, for any $p>0$.
	Hence, for any $U_1, U_2$
	\begin{align*}
		\mathbb{E}_{U_1, U_2}\left[ \left(\frac{\tilde{\mu}^{(1)}_{1}}{n^{1/2}} \wedge 1\right) \left(\frac{\tilde{\mu}^{(2)}_{1}}{n^{1/2}} \wedge 1\right) \mathds{1}_{\{\tilde{E}_n^c\}} \Big| \mathcal{D}^{\bullet} = \infty\right] \le n^{-p+1} \le n^{-\tilde{p}}.
	\end{align*}
For the other term, we apply the Cauchy-Schwarz inequality and using Lemma~\ref{lemma:TailConditioning} we obtain
\begin{align*}
	\mathbb{E}_{U_1, U_2}&\left[ \left(\frac{\tilde{\mu}^{(1)}_{1}}{n^{1/2}} \wedge 1\right) \left(\frac{\tilde{\mu}^{(2)}_{1}}{n^{1/2}} \wedge 1\right) \mathds{1}_{\{\tilde{E}_n\}} \Big| \mathcal{D}^{\bullet} = \infty\right] \\&\le c_1 \mathbb{E}_{U_1}\left[ \left(\frac{\tau^{(1)}_{n^{\varepsilon}}}{n^{1/2}} \wedge 1\right)^2 \Big| D^{v^*} = \infty\right]^{1/2} \mathbb{E}_{U_2}\left[ \left(\frac{\tau^{(2)}_{n^{\varepsilon}}}{n^{1/2}} \wedge 1\right)^2 \Big| D^{v^*} = \infty\right]^{1/2}.
\end{align*}
    Finally, on $\tilde{E}_n$ we have $\mu_1^{(1)} \le \tau^{(1)}_{n^{\varepsilon}}$ which is a sum of $n^{\varepsilon}$ i.i.d.\ random variables with finite second moment. Let us call these random variables $X_1, \dots, X_{n^{\varepsilon}}$ and their expectation $E$ for sake of notational simplicity. Hence, we get the upper bound
    \begin{align*}
    	E\left[ \left(\sum_{i = 1}^{n^{\varepsilon}}\left(\frac{X_i}{n^{1/2}} \wedge 1\right) \right)^2\right] 
    	&\le \frac{1}{n} \sum_{i = 1}^{n^{\varepsilon}} E[X_i^2] + \frac{2}{n}\sum_{i<j = 1}^{n^{\varepsilon}} E[X_i] E[X_j]\\
    	&\le n^{-1+\varepsilon} + 2 n^{-1+2\varepsilon} \le C n^{-1+2\varepsilon}.
    \end{align*}
\end{proof}
\begin{lemma}\label{lemma:TailConditioning}
	There exists an absolute constant $\rho>0$ such that, uniformly over all $U_1, U_2 \in \mathbb{V}_d$ we have that, 	for all $p \in (0, 2]$
	\begin{equation*}
		\mathbb{E}_{U_1, U_2}\left[ (\tau_{1}^{(1)})^p \, \big| \, \mathcal{D}^\bullet = \infty \right] \le \rho \mathbb{E}_{U_1}\left[ \tau_{1}^p \, \big| \, D^{v^*} = \infty \right].
	\end{equation*}
\end{lemma}
\begin{proof}
	For brevity's sake we denote $(\tau_{1}^{(1)}) = \tau_{1}$. This result follows by conditioning and the fact that $\{\mathcal{D}^\bullet = \infty\} \subseteq \{D^{v^*} = \infty\}$. Using Proposition~\ref{prop:ShortLevels}
	\begin{align*}
		\mathbb{E}_{U_1, U_2}\left[ \tau_{1}^p \, \big| \, \mathcal{D}^\bullet = \infty \right] 
		&= \frac{\mathbb{E}_{U_1, U_2}\left[ \tau_{1}^p \mathds{1}_{\{\mathcal{D}^\bullet  = \infty\}} \right]}{\mathbb{P}_{U_1, U_2}(\mathcal{D}^\bullet  = \infty)}\le \eta^{-1} \mathbb{E}_{U_1, U_2}\left[ \tau_{1}^p \mathds{1}_{\{D^{v^*}  = \infty\}} \right]\\
		& \le \eta^{-1} \mathbb{E}_{U_1, U_2}\left[ \tau_{1}^p \, \big| \, D^{v^*} = \infty \right].
    \end{align*}
Finally, we observe that the marginal of $X^{(1)}$ under $\mathbb{P}_{U_1, U_2}$ is the distribution of $X$ under $\mathbb{P}_{U_1}$.
\end{proof}

The following two results deal with the slightly different behaviour that the first regeneration time and the first regeneration level have w.r.t.\ all other regeneration periods.

\begin{lemma}\label{lemma:Extra1}There exists a positive constant $C>0$, such that
	\begin{equation*}
		\mathbf{E}\left[ E_{0, 0}^\omega\left[ \left(\frac{\mu^{(1)}_{1} - \tau^{(1)}_1}{n^{1/2}} \wedge 1\right) \left(\frac{\mu^{(2)}_{1} - \tau^{(2)}_1}{n^{1/2}} \wedge 1\right) \right]\right]\le Cn^{2\varepsilon - 1}.
	\end{equation*}
\end{lemma}
\begin{proof}
	The proof follows the same lines as the proof of Lemma~\ref{lemma:TailJRL}, we only sketch it. Let $\tilde{E}_n \coloneqq \{\mu_1^{(1)} \le \tau^{(1)}_{n^{\varepsilon}}, \mu_1^{(2)} \le \tau^{(2)}_{n^{\varepsilon}}\}$. As in Lemma~\ref{lemma:TailJRL}, we get $\mathbb{P}_{0, 0}(\tilde{E}_n) \ge \mathbb{P}_{0, 0}(E_n) \ge 1 -n^{-p}$.
	Hence,
	\begin{align*}
		\mathbb{E}_{0, 0}\left[\left(\frac{\mu^{(1)}_{1} - \tau^{(1)}_1}{n^{1/2}} \wedge 1\right) \left(\frac{\mu^{(2)}_{1} - \tau^{(2)}_1}{n^{1/2}} \wedge 1\right) \mathds{1}_{\{\tilde{E}_n^c\}} \right] \le n^{-p+1} \le n^{-\tilde{p}}.
	\end{align*} 
    Furthermore, 
\begin{align*}
	\mathbb{E}_{0, 0}&\left[\left(\frac{\mu^{(1)}_{1} - \tau^{(1)}_1}{n^{1/2}} \wedge 1\right) \left(\frac{\mu^{(2)}_{1} - \tau^{(2)}_1}{n^{1/2}} \wedge 1\right) \mathds{1}_{\{\tilde{E}_n^c\}} \right] \\
	&\le c_1 \mathbb{E}_{0}\left[ \left(\frac{\tau^{(1)}_{n^{\varepsilon}} - \tau^{(1)}_{1}}{n^{1/2}} \wedge 1\right)^2 \right]^{1/2} \mathbb{E}_{0}\left[ \left(\frac{\tau^{(2)}_{n^{\varepsilon}} - \tau^{(2)}_{1}}{n^{1/2}} \wedge 1\right)^2\right]^{1/2}.
\end{align*}
We know from their definition that $\tau^{(1)}_{n^{\varepsilon}} - \tau^{(1)}_{1} = \sum_{i = 2}^{n^\varepsilon} \tilde{\tau}_i$ is a sum of $n^\varepsilon$ i.i.d.\ random variables with second moment. Following the final steps of the proof of Lemma~\ref{lemma:TailJRL} verbatim gives the result.
\end{proof}

\begin{lemma}\label{lemma:Extra2}
	We have that
	\begin{equation*}
		\sup_{0 \le z\cdot v^* \le n^{\varepsilon}} \mathbb{E}_0[ T_{z\cdot v^*}^2| D^{v^*} \ge T_{z\cdot v^*}] \le Cn^{\varepsilon}.
	\end{equation*}
\end{lemma}
\begin{proof}
	We show the proof for $z\cdot v^* = n^\varepsilon$, the proof for all other values in the range considered follows the same steps. We begin by noticing the fact that $\{D^{v^*} \ge T_{n^{\varepsilon}}\}$ is measurable w.r.t.\ 
	\begin{equation*}
		\sigma\left(\{X_n^{(1)}\}_{n = 0}^{T_{n^{\varepsilon}}} \cup \{\omega_x \colon x \in \{X_n^{(1)}\}_{n = 0}^{T_{n^{\varepsilon}}}\}\right).
	\end{equation*}
    We can write
    \begin{align*}
    	\mathbb{E}_0[ T_{n^{\varepsilon}}^2| D^{v^*} \ge T_{n^{\varepsilon}}] & = \frac{\mathbb{E}_0[ T_{n^{\varepsilon}}^2 \mathds{1}_{\{D^{v^*} \ge T_{n^{\varepsilon}}\}} ]}{\mathbb{P}_0(D^{v^*} \ge T_{n^{\varepsilon}})}.
    \end{align*}
    From inclusion of the events we have that $\mathbb{P}_0(D^{v^*} \ge T_{n^{\varepsilon}}) \ge \mathbb{P}_0(D^{v^*} = \infty)$. We obtain 
	\begin{align*}
   		\mathbb{E}_0[ T_{n^{\varepsilon}}^2| D^{v^*} \ge T_{n^{\varepsilon}}] 
   		\le \frac{\mathbb{E}_0[ T_{n^{\varepsilon}}^2\mathds{1}_{\{D^{v^*} \ge T_{n^{\varepsilon}}\}}]}{\mathbb{P}_0(D^{v^*} = \infty)}= \frac{\mathbb{E}_0[ T_{n^{\varepsilon}}^2\mathds{1}_{\{D^{v^*} \ge T_{n^{\varepsilon}}\}} \mathbb{P}_{X_{T_{n^{\varepsilon}}}}(D^{v^*} = \infty)]}{\mathbb{P}_0(D^{v^*} = \infty)^2}.
   	\end{align*}
    We note $\mathbb{P}_{0}(D^{v^*} = \infty)$ is invariant under translation of the starting point and hence the step above is justified since $\mathbb{P}_{X_{T_{n^{\varepsilon}}}}(D^{v^*} = \infty)$ is a constant random variable. We observe that $\mathbb{P}_{X_{T_{n^{\varepsilon}}}}(D^{v^*} = \infty)$ depends only on the environment to the right of $X_{T_{n^{\varepsilon}}}$. Thus,
    \begin{equation*}
    	\mathbb{E}_0[ T_{n^{\varepsilon}}^2\mathds{1}_{\{D^{v^*} \ge T_{n^{\varepsilon}}\}} \mathbb{P}_{X_{T_{n^{\varepsilon}}}}(D^{v^*} = \infty)] = \mathbb{E}_0[ T_{n^{\varepsilon}}^2\mathds{1}_{\{D^{v^*} \ge T_{n^{\varepsilon}}, D^{v^*}\circ \theta_{T_{n^{\varepsilon}}} = \infty\}} ].
    \end{equation*}
Noting that $\{D^{v^*} \ge T_{n^{\varepsilon}},  D^{v^*}\circ \theta_{T_{n^{\varepsilon}}}\}\subseteq  \{D^{v^*} = \infty\}$ and recalling that $\mathbb{P}_0(D^{v^*} = \infty)\ge \eta > 0$ we are finally able to write
\begin{equation*}
	\mathbb{E}_0[ T_{n^{\varepsilon}}^2| D^{v^*} \ge T_{n^{\varepsilon}}] \le \frac{1}{\eta} \mathbb{E}_0[ T_{n^{\varepsilon}}^2| D^{v^*} =\infty].
\end{equation*}
To end the proof, one observes that $T_{n^{\varepsilon}}^2\le \tau_{n^\varepsilon}$ and follows once again the final steps of the proof of Lemma~\ref{lemma:TailJRL}.
\end{proof}

\subsubsection*{Acknowledgments}

I thank Noam Berger for several useful discussions on \cite{Berger_Zeitouni} and this topic in general. I thank U.\ De Ambroggio and T.\ Lions for helping developing some of the ideas presented in this work during several discussions.

\printbibliography
\end{document}